\documentclass[a4paper,reqno]{amsart}
\usepackage[utf8]{inputenc}
\usepackage[T1]{fontenc}
\usepackage{amssymb,mathabx,xy}
\usepackage{amsmath,amsthm}
\xyoption{all}
\usepackage{hyperref}
\usepackage{mathbbol}

\newenvironment{tfae}
{
	\begin{enumerate}}
		{\end{enumerate}}

\newtheorem{theorem}[subsection]{Theorem}
\newtheorem{lemma}[subsection]{Lemma}
\newtheorem{proposition}[subsection]{Proposition}
\newtheorem{corollary}[subsection]{Corollary}

\theoremstyle{definition}

\newtheorem{definition}[subsection]{Definition}

\def\noproof{\hfill\qed}

\numberwithin{equation}{section}

\DeclareMathOperator{\Ker}{Ker}
\DeclareMathOperator{\Img}{Im}
\DeclareMathOperator{\ab}{ab}
\DeclareMathOperator{\length}{length}

\renewcommand{\ll}{\mathfrak{l}}
\newcommand{\mm}{\mathfrak{m}}
\newcommand{\nn}{\mathfrak{n}}
\renewcommand{\gg}{\mathfrak{g}}
\newcommand{\hh}{\mathfrak{h}}

\newcommand{\rr}{\mathfrak{r}}
\newcommand{\pp}{\mathfrak{p}}

\newcommand{\K}{\mathbb{K}}
\newcommand{\N}{\mathbb{N}}

\newdir{>}{{}*:(1,-.2)@^{>}*:(1,+.2)@_{>}}
\newdir{<}{{}*:(1,+.2)@^{<}*:(1,-.2)@_{<}}

\begin{document}

\author{Guram Donadze}
\author{Tim Van~der Linden}

\email[G.~Donadze]{gdonad@gmail.com}
\email[T.~Van~der Linden]{tim.vanderlinden@uclouvain.be}

\address[G.~Donadze]{Institute of Cybernetics of Georgian Tech Univ, Sandro Euli Str: 5, Tbilisi 0186 Georgia}
\address[T.~Van~der Linden]{Institut de Recherche en Math\'ematique et Physique, Universit\'e catholique de Louvain, chemin du cyclotron~2 bte~L7.01.02, 1348 Louvain-la-Neuve, Belgium}
\thanks{The first author was financially supported by Shota Rustaveli National Science Foundation of Georgia, grant FR-22-199. The second author is a Senior Research Associate of the Fonds de la Recherche Scientifique--FNRS}

\title[Schur- and Baer-type theorems for Lie and Leibniz algebras]{Schur- and Baer-type theorems\\ for Lie and Leibniz algebras}

\begin{abstract}
	The aim of this article is to obtain variations on the classical theorems of Schur and Baer on finiteness of commutator subgroups, valid in the contexts of Lie algebras and Leibniz algebras over a field. Using non-abelian tensor products and exterior products, we prove Schur's Theorem for finitely generated Leibniz algebras, both Schur's Theorem and Baer's Theorem for finitely generated Lie algebras, and a version of these theorems for finitely presented Lie algebras.
\end{abstract}



\maketitle

\section{Introduction}\label{S:In}

Schur's Theorem~\cite{Schur1904,DKP2015} states that, for a group $G$ and $C \leq Z(G)$, if $G/C$ is a finite group, then the derived subgroup $G'$ is finite. Baer~\cite{Baer1952} extended this by proving that, if for a natural number $i$ the quotient $G/Z_i(G)$ is finite, then $\gamma_{i+1}(G)$ is finite. This result is commonly known as Baer's Theorem.

Several kinds of variations of these results, for groups and for other types of algebraic objects, exist in the literature. Closely related to our current work are, for instance, a version of Schur's theorem for Lie algebras obtained in~\cite{Moneyhun}: here the author proves that if $\ll/Z(\ll)$ is finite-dimensional, then so is the commutator ideal $\ll'$. An analogue of Baer's theorem for Lie algebras occurs in~\cite{SEA}: if $\ll/Z_i(\ll)$ is finite-dimensional, then so is $\gamma_{i+1}(\ll)$.

Our ultimate aim is to work towards a categorical approach to such results, using the tensor product of~\cite{dMVdL19.3}, which captures the non-abelian tensor product introduced by Brown and Loday \cite{Brown-Loday,BL1984}, as well as, for instance, Ellis's version of it for Lie algebras~\cite{El0}.

The present article is a preliminary step in this direction. We use techniques from homological algebra, with in particular the non-abelian tensor products (as well as the closely related exterior products) of Lie and Leibniz algebras respectively introduced in~\cite{El0} and in~\cite{Gn99} to prove Schur's Theorem for finitely generated Leibniz algebras, both Schur's Theorem and Baer's Theorem for finitely generated Lie algebras, and a version of these theorems for finitely presented Lie algebras.

The text is structured as follows. In Section~\ref{S:Leibniz} we consider the case of Leibniz algebras. After the technically important Theorem~\ref{theorem1}, here the main result is a Schur-type theorem, which says that given a finitely generated Leibniz algebra $\mm$, the commutator $[\mm, \mm]$ is finitely generated if and only if for every central extension $0\to \rr \to \ll \to \mm \to 0$ of Leibniz algebras, the commutator $[\ll, \ll]$ is finitely generated (Theorem~\ref{ThSchurLeibniz}). 

This is used to obtain two kinds of Schur/Baer-type result for Lie algebras. In Section~\ref{S:LieFinGen} we consider the case of finitely generated Lie algebras. We first prove Theorem~\ref{ThmTensorGamma}, which says that if $\mm$ is a finitely generated Lie algebra and $i$ a natural number, then $\gamma_{i+1}(\mm)$ is finitely generated if and only if so is $\mm^{\otimes i+1}$. This is then used in Theorem~\ref{theorem3.11} stating that if $\mm$ is a finitely generated Lie algebra and $i$ a natural number, then $\gamma_{i+1}(\mm)$ is finitely generated if and only if for every $i$-central extension $0 \to \rr \to \ll \to \mm \to 0$ of Lie algebras, $\gamma_{i+1}(\ll)$ is finitely generated. Section~\ref{S:LieFinPres} treats similar results for \emph{finitely presented} Lie algebras: Theorem~\ref{ThmTensorGammaFinPres} and Theorem~\ref{Thm4.3}, respectively.

Throughout the text, we let $\K$ be a field, and consider Lie and Leibniz algebras with respect to this field.

\section{The Case of Leibniz Algebras}\label{S:Leibniz}

\subsection{Leibniz actions.}

Actions correspond to split extensions via a semidirect product construction. In the setting of Leibniz	algebras, it is convenient to view them as follows~\cite{Gn99}.

\begin{definition}
	Let $\mm$ and $\nn$ be Leibniz algebras. A \emph{Leibniz action} of $\mm$ on
	$\nn$ is a couple of bilinear maps $\mm \times \nn \to \nn\colon (m, n)\mapsto {}^mn$ and
	$\nn \times \mm \to \nn\colon(n, m)\mapsto n^m$ satisfying
	\begin{align*}
		^{[m, m']}n  & ={}^m({}^{m'}n)+({}^mn)^{m'} & {}^m[n, n'] & =[{}^mn, n']-[{}^mn', n] \\
		n^{[m, m']}  & =(n^m)^{m'}-(n^{m'})^{m}     & [n, n']^m   & =[n^m, n']+[n, n'^m]     \\
		^m({}^{m'}n) & =-^m(n^{m'})                 & [n, {}^mn'] & =-[n, n'^{m}]
	\end{align*}
	for all $m$, $m'\in \mm$, $n$, $n'\in \nn$.
\end{definition}

The non-abelian tensor product (see~\ref{Nonab Tensor} below) is defined for two objects acting on each other. When such actions are compatible in the sense of the following definition, this means~\cite{Gn99} that they are induced by a pair of Leibniz crossed modules $\mm\to\gg$, $\nn\to\gg$ with common codomain $\gg$. This interpretation of what is a pair of compatible actions happens to be a general categorical fact (Theorem 4.11 in~\cite{dMVdL19}).

\begin{definition}[\cite{Gn99}]
	Let $\mm$ and $\nn$ be two Leibniz algebras acting on each other. We say these actions are \emph{compatible} if
	\begin{align*}
		^{({}^mn)}m'   & = [m^n, m']   & {}^{({}^nm)}n' & = [n^m, n']    \\
		^{(n^m)}m'     & = [{}^nm, m'] & {}^{(m^n)}n'   & = [{}^mn, n']  \\
		m^{({}^{m'}n)} & = [m, m'^n]   & n^{({}^{n'}m)} & = [n, n'^{m}]  \\
		m^{(n^{m'})}   & =[m, {}^nm']  & n^{(m^{n'})}   & = [n, {}^mn']
	\end{align*}
	for $m$, $m' \in \mm$ and $n$, $n' \in \nn$.
\end{definition}

\subsection{The non-abelian tensor product.}\label{Nonab Tensor}
Let $\mm$ and $\nn$ be Leibniz algebras with mutual actions on one another.
The \emph{non-abelian tensor product} of $\mm$ and $\nn$, denoted by $\mm \star \nn$, is defined in \cite{Gn99} to be
the Leibniz algebra generated by the symbols $m*n$ and $n*m$, for all $m\in \mm$ and $n\in \nn$, subject to the relations in Table~\ref{Table Tensor Relations}.
\begin{table}
	\tiny
	\begin{align*}
		(1\textrm{a})\quad & k(m*n)=km*n=m*kn                         & (1\textrm{b})\quad & k(n*m)=kn*m=n*km                         \\
		(2\textrm{a})\quad & (m+m')*n=m*n+m'*n                        & (2\textrm{b})\quad & (n+n')*m=n*m+n'*m                        \\
		(2\textrm{c})\quad & m*(n+n')=m*n+m*n'                        & (2\textrm{d})\quad & n*(m+m')=n*m+n*m'                        \\
		(3\textrm{a})\quad & m*[n, n']=m^n*n'-m^{n'}*n                & (3\textrm{b})\quad & n*[m, m']=n^m*m'-n^{m'}*m                \\
		(3\textrm{c})\quad & [m, m']*n={}^mn*m'-m*n^{m'}              & (3\textrm{d})\quad & [n, n']*m={}^nm*n'-n*m^{n'}              \\
		(4\textrm{a})\quad & m*{}^{m'}n=-m*n^{m'}                     & (4\textrm{b})\quad & n*{}^{n'}m=-n*m^{n'}                     \\
		(5\textrm{a})\quad & m^n*{}^{m'}n'=[m*n, m'*n']={}^mn*m'^{n'} & (5\textrm{b})\quad & ^nm*n'^{m'}=[n*m, n'*m']=n^m*{}^{n'}m'   \\
		(5\textrm{c})\quad & m^n*n'^{m'}=[m*n, n'*m']={}^mn*{}^{n'}m' & (5\textrm{d})\quad & ^nm*{}^{m'}n'=[n*m, m'*n']=n^m* m'^{n'}
	\end{align*}
	\caption{Relations of $\mm\star\nn$, valid for all $k\in \K$, $m$, $m'\in \mm$, $n$, $n'\in \nn$.}\label{Table Tensor Relations}
\end{table}
There are induced homomorphisms of Leibniz algebras $\tau_{\mm} \colon \mm \star \nn \to \mm$ and $\tau_{\nn} \colon \mm \star \nn \to \nn$ where $\tau_{\mm}(m * n) = m^n$, $\tau_{\mm}(n * m) = {}^nm$, $\tau_{\nn}(m * n) = {}^mn$ and $\tau_{\nn}(n * m) = n^m$.

\begin{theorem} \label{theorem1} Let $\mm$ and $\nn$ be finitely generated Leibniz algebras acting on each other
	compatibly. Then, $\mm \star \nn$ is finitely generated iff $\tau_{\mm}(\mm \star \nn)$ and
	$\tau_{\nn}(\mm \star \nn)$ are finitely generated.
\end{theorem}
\begin{proof} We will prove the ``if'' part of the theorem. The other part is trivial.
	
	We have natural morphisms $(\tau_{\mm}(\mm \star \nn)) \star \nn \to \mm \star \nn$ and
	${\mm \star (\tau_{\nn}(\mm \star \nn)) \to \mm \star \nn}$, whose images we will respectively denote by
	$(\tau_{\mm}(\mm \star \nn))\star \nn$ and $\mm \star (\tau_{\nn}(\mm \star \nn))$. Since the actions are compatible, $\tau_{\mm}(\mm \star \nn)$ and $\tau_{\nn}(\mm \star \nn)$ are ideals of $\mm$ and $\nn$,
	respectively. We have the following exact sequence.
	$$
		0 \to \big((\tau_{\mm}(\mm \star \nn)) \star \nn\big) \vee \big(\mm \star (\tau_{\nn}(\mm \star \nn))\big) \to
		\mm \star \nn \to \frac{\mm}{\tau_{\mm}(\mm \star \nn)} \star \frac{\nn}{\tau_{\nn}(\mm \star \nn)} \to 0
	$$
	Let $M_1=\{m_1, \ldots, m_k\}$ be a set of generators of $\mm$ and $N_1=\{n_1, \ldots, n_l\}$ a set of generators of $\nn$.
	Since $\mm/\tau_{\mm}(\mm \star \nn)$ and $\nn/\tau_{\nn}(\mm \star \nn)$ act on each other trivially,
	$\big(\mm/\tau_{\mm}(\mm \star \nn)\big)\star\big(\nn/\tau_{\mm}(\mm \star \nn)\big)$ will be generated
	by $\big(m_i+\tau_{\mm}(\mm \star \nn)\big)*\big(n_j+\tau_{\nn}(\mm \star \nn)\big)$ and
	$\big(n_j+\tau_{\nn}(\mm \star \nn)\big)*\big(m_i+\tau_{\mm}(\mm \star \nn)\big)$, for $i=1, \ldots , k$ and $j=1, \ldots, l$.
	Taking into account the exact sequence mentioned above, it suffices to show that
	\begin{enumerate}
		\item there is a finitely generated $\gg$ such that
		      $(\tau_{\mm}(\mm \star \nn))\star \nn \subset \gg \subseteq \mm \star \nn$;
		\item there is a finitely generated $\hh$ such that
		      $\mm \star (\tau_{\nn}(\mm \star \nn)) \subset \hh \subseteq \mm \star \nn$.
	\end{enumerate}
	We will prove only (1). The proof of (2) is similar to that of (1) and left to the reader.
	
	Let $M_2=\{x_1^{y_1}, \ldots, x_s^{y_s}, {}^{z_1}w_1, \ldots, {}^{z_t}w_t\}$ be a set of generators for $\tau_{\mm }(\mm \star \nn)$.
	Define a map $\sigma\colon M_2 \to \nn$ by
	$\sigma (x_i^{y_i})={}^{x_i}y_i$, $\sigma ({}^{z_j}w_j)=z_j^{w_j}$. Since the actions are compatible, we have:
	\begin{align}\label{new1}
		[m, b]=m^{\sigma(b)} \qquad \text{and} \qquad [b, m]={}^{\sigma(b)}m ,
	\end{align}
	for each $m\in \mm$ and $b\in M_2$. We need one more map, $\alpha\colon M_2 \to \mm \star \nn$, determined by
	$\alpha (x_i^{y_i})=x_i*y_i$, $\alpha ({}^{z_j}w_j)=z_j*w_j$. Consider the following subset of $\mm \star \nn$:
	$$
		X=\{m*n \mid n\in N_1\cup \sigma (M_2), \: m\in M_2\} \cup \{n*m \mid n\in N_1\cup \sigma (M_2), \: m\in M_2\} \cup
		\alpha (M_2).
	$$
	Denote by $\gg$ the subalgebra of $\mm \star \nn$, generated by the elements of $X$.
	We will prove that $(\tau_{\mm}(\mm \star \nn )) \star \nn \subset \gg$.
	By the defining relations of the non-abelian tensor product it suffices to show that $m*n$, $n*m \in \gg$,
	for each $m\in \tau_{\mm}(\mm \star \nn)$ and $n\in N_1$. We prove this by induction, as follows.
	
	Let $m\in \tau_{\mm}(\mm \star \nn)$ be an element which can be factored into a Leibniz product of elements from $M_2$.
	Then, for $i\in \N$, we write $\length(m)\leq i$, if $m$ can be factored into a Leibniz product of $i$ (not necessarily distinct)
	elements from $M_2$. Moreover, if such $m$ can not be factored into a Leibniz product of fewer
	elements from $M_2$, then we write $\length(m)=i$.
	
	Let $b$, $d\in M_2$ and $n\in \nn$. Then,
	\begin{align*}
		[b,d]*n  & ={}^{b}n*d-b*n^d = [b*n, \alpha(d)]-[\alpha(b), n*d], \\
		n*[b, d] & = n^b *d-n^d*b = [n*b, \alpha(d)]-[n*d, \alpha(b)].
	\end{align*}
	This implies that if $\length(m) \leq 2$, then $m*n$, $n*m\in \gg$ for all $n\in N_1\cup \sigma(M_2)$. This is the base step of the induction. 
	Now assume the induction hypothesis that $m*n$, $ n*m\in \gg$ for all $m\in \tau_{\mm}(\mm \star \nn)$ and $n$ such that
	$n\in N_1\cup \sigma(M_2)$ and $\length(m) \leq i-1$, where $i\geq 3$.
	Suppose that $m'\in \tau_{\mm}(\mm \star \nn)$, $\length(m')=i$. Then $m'$ is one of $[[c, d],b]$, $[[d, c], b]$, $[b, [c,d]]$ or $[b, [d,c]]$,	where $b$, $d\in M_2$, $c\in \tau_{\mm}(\mm \star \nn)$ and $\length(c)=i-2$.
	By (\ref{new1}) and by the defining relations of the non-abelian tensor product, we have:
	\begin{align*}
		[[c, d], b]* n & ={}^{[c, d]}n*b - [c, d]*n^{b} = [[c,d]*n, \alpha(b)]-c^{\sigma(d)}* n^b \\
		               & =[[c,d]*n, \alpha(b)]-[c*\sigma(d), n*b];
	\end{align*}
	\begin{align*}
		[[d, c], b]* n & ={}^{[d, c]}n*b - [d, c]*n^{b} = [[d,c]*n, \alpha(b)]- {}^{\sigma(d)}c* n^b \\
		               & =[[d,c]*n, \alpha(b)]-[\sigma(d)*c, n*b];
	\end{align*}
	\begin{align*}
		[b, [c,d]]* n & ={}^{b}n*[c,d] - b*n^{[c,d]}={}^{b}n*c^{\sigma(d)} - b*n^{[c,d]} \\
		              & =[b*n, c*\sigma(d)]-[\alpha(b), n*[c,d]];
	\end{align*}
	\begin{align*}
		[b, [d,c]]* n & ={}^{b}n*[d,c] - b*n^{[d,c]}={}^{b}n*{}^{\sigma(d)}c - b*n^{[d,c]} \\
		              & =[b*n, \sigma(d)*c]-[\alpha(b), n*[d,c]];
	\end{align*}
	\begin{align*}
		n*[[c,d], b] & = n^{[c,d]}*b-n^{b}*[c,d]=[n*[c,d], \alpha(b)]-n^{b}*c^{\sigma(d)} \\
		             & =[n*[c,d], \alpha(b)]-[n*b,c*\sigma(d)];
	\end{align*}
	\begin{align*}
		n*[[d,c], b] & = n^{[d,c]}*b-n^{b}*[d,c]=[n*[d,c], \alpha(b)]-n^{b}* {}^{\sigma(d)}c \\
		             & =[n*[d,c], \alpha(b)]-[n*b, \sigma(d)* c];
	\end{align*}
	\begin{align*}
		n*[b, [c, d]] & = n^b*[c, d]-n^{[c, d]}*b= n^b*c^{\sigma(d)}-n^{[c, d]}*b \\ 
		              & = [n*b, c*\sigma(d)]-[n*[c,d], \alpha(b)];
	\end{align*}
	\begin{align*}
		n*[b, [d, c]] & = n^b*[d, c]-n^{[d, c]}*b=n^b*{}^{\sigma(d)}c-n^{[d, c]}*b \\
		              & =[n*b, \sigma(d)*c]- [n*[d, c], \alpha(b)].
	\end{align*}
	Since $\length(c)$, $\length([c,d])$, $\length([d,c])\leq i-1$,
	from the induction hypothesis it follows that $m'*n$, $n*m'\in \gg$ for each $n\in N_1\cup \sigma(M_2)$ and
	$\length(m')=i$. This proves our claim.
\end{proof}

\begin{corollary} Let $\mm$ be a finitely generated Leibniz algebra. Then $\mm \star \mm$ is finitely generated
	iff the Leibniz commutator $[\mm, \mm]$ is finitely generated.\noproof
\end{corollary}

\begin{theorem}\label{ThSchurLeibniz} Let $\mm$ be a finitely generated Leibniz algebra. Then the following are equivalent:
	\begin{tfae}
		\item $[\mm, \mm]$ is finitely generated;
		\item there exists a central extension $0\to \rr \to \ll \to \mm \to 0$ of Leibniz algebras, where $[\ll, \ll]$ is finitely generated;
		\item for every central extension $0\to \rr \to \ll \to \mm \to 0$ of Leibniz algebras, $[\ll, \ll]$ is finitely generated.
	\end{tfae}
\end{theorem}
\begin{proof} The implications (iii)$\Rightarrow$(ii)$\Rightarrow$(i) are obvious, so let us show that (i) implies (iii). Consider a central extension as in (iii). We have a short exact sequence
	of Leibniz algebras
	\begin{align}\label{new2}
		0\to (\rr \star \ll) \vee (\ll \star \rr) \to \ll \star \ll \to \mm \star \mm \to 0 ,
	\end{align}
	where $\rr \star \ll$ and $\ll \star \rr$ are identified with their images into $\ll \star \ll$.
	Observe that $\tau_\mm (\mm \star \mm)=[\mm, \mm]$. Moreover, since the extension $0\to \rr \to \ll \to \mm \to 0$
	is central, we have that $\tau_\mm ((\rr \star \ll) \vee (\ll \star \rr))=0$. By (\ref{new2}) we get a surjection $\mm \star \mm \to [\ll, \ll]$. By the previous corollary, $\mm \star \mm$ is finitely
	generated. This proves the theorem.
\end{proof}

\section{On Finitely Generated Lie Algebras}\label{S:LieFinGen}

\subsection{Lie actions}
In the case of Lie algebras, the description of an action and of the tensor product of a pair of algebras acting on each other (compatibly) simplifies as follows~\cite{El0}.

\begin{definition}
	Let $\mm$ and $\nn$ be Lie algebras. A \emph{Lie action} of $\mm$ on
	$\nn$ is a bilinear map $\mm \times \nn \to \nn\colon (m, n)\mapsto {}^mn$, satisfying
	\begin{align*}
		^{[m, m']}n={}^m({}^{m'}n)-{}^{m'}({}^{m}n)
		\qquad\text{and}\qquad
		{}^m[n, n']=[{}^mn, n']+[n, {}^mn']
	\end{align*}
	for each $m$, $m'\in \mm$, $n$, $n'\in \nn$.
\end{definition}

\begin{definition}
	Let $\mm$ and $\nn$ be two Lie algebras acting on each other. We say these actions are \emph{compatible} when
	\begin{align*}
		^{({}^mn)}m' = [m', {}^nm]
		\qquad\text{and}\qquad
		{}^{({}^nm)}n' = [n', {}^mn]
	\end{align*}
	for all $m$, $m' \in \mm$ and $n$, $n' \in \nn$.
\end{definition}

\begin{definition}
	Let $\mm$ and $\nn$ be Lie algebras with mutual actions on one another.
	The \emph{non-abelian tensor product} of $\mm$ and $\nn$, denoted by $\mm \otimes \nn$, is defined to be
	the Lie algebra generated by the symbols $m\otimes n$, for all $m\in \mm$ and $n\in \nn$, subject to
	\begin{align*}
		k(m\otimes n)              & =km\otimes n=m\otimes kn             \\
		(m+m')\otimes n            & = m\otimes n+m'\otimes n             \\
		m\otimes (n+n')            & = m\otimes n + m\otimes n'           \\
		m\otimes [n, n']           & ={}^{n'}m\otimes n-^nm\otimes n'     \\
		[m, m']\otimes n           & = m\otimes{}^{m'}n - m'\otimes {}^mn \\
		[m\otimes n, m'\otimes n'] & =-{}^nm \otimes {}^{m'}n'
	\end{align*}
	for each $k\in \K$, $m$, $m'\in \mm$, $n$, $n'\in \nn$.
\end{definition}

There are induced morphisms of Lie algebras $\mu\colon \mm \otimes \nn \to \mm$ and ${\nu\colon \mm \otimes \nn \to \nn}$, defined on generators by
$\mu (m \otimes n)=-{}^nm$ and ${\nu(m \otimes n)= {}^mn}$.

\begin{theorem} \label{theorem2} Let $\mm$ and $\nn$ be finitely generated Lie algebras acting on each other
	compatibly. Then $\mm \otimes \nn$ is finitely generated iff $\mu (\mm \otimes \nn)$ and $\nu (\mm \otimes \nn)$ are finitely generated.
\end{theorem}
\begin{proof} We have a surjective Leibniz algebra morphism $\mm \star \nn \to \mm \otimes \nn$ defined on generators
	by $m*n \mapsto m\otimes n$ and $n*m \mapsto -m\otimes n$, for $m\in \mm$ and $n\in \nn$. Hence, the theorem results from Theorem \ref{theorem1}.
\end{proof}

\subsection{Crossed modules of Lie algebras}
A \emph{crossed module} is a Lie homomorphism $\partial\colon \mm \to \pp$ together with an action of $\pp$ on $\mm$
such that
\begin{align*}
	\partial ({}^pm)=[p, \partial(m)]
	\qquad\text{and}\qquad
	{}^{\partial (m)}m'=[m, m']
\end{align*}
where $m$, $m' \in \mm$, $p\in \pp$.

\begin{proposition} \label{prop1}(\cite{El0})
	Let $\mm$ and $\nn$ be Lie algebras acting on each other compatibly. There are Lie actions
	of $\mm$ and $\nn$ on $\mm \otimes \nn$ given by
	$$
		^m(m'\otimes n')= [m, m']\otimes n' + m'\otimes{}^{m}n'
		\qquad\text{and}\qquad
		{}^n(m'\otimes n')= {}^nm'\otimes n' + m'\otimes [n, n'],
	$$
	where $m$, $m'\in \mm$, $n$, $n'\in \nn$. Moreover, $\mu\colon \mm \otimes \nn \to \mm$ and $\nu\colon \mm \otimes \nn \to \nn$
	are crossed modules.\noproof
\end{proposition}

Let $\mm$ and $\nn$ be Lie algebras acting on each other compatibly. It is easy to check that
the action of $\mm \otimes \nn$ on $\nn$, given by $^{\omega}n=[\nu(\omega ), n]$, for each
$\omega \in \mm \otimes \nn$, $n\in \nn$, is a well-defined Lie action.

\begin{lemma} The mutual actions of $\mm \otimes \nn$ and $\nn$ are compatible.
\end{lemma}
\begin{proof} We need to check that
	$$
		^{^{(m\otimes n)}n'}(m_1\otimes n_1)=[m_1\otimes n_1,{}^{n'}(m\otimes n)]
		\qquad\text{and}\qquad
		^{^{n_1}(m\otimes n)}n_2= [n_2, {}^{m\otimes n}n_1]
	$$
	for each $m$, $m_1\in \mm$, $n$, $n_1$, $n_2$, $n'\in \nn$. We have
	\begin{align*}
		^{^{(m\otimes n)}n'}(m_1\otimes n_1) & = {}^{[{}^mn, n']}(m_1\otimes n_1)                                                              \\
		                                     & = {}^{[{}^mn, n']}m_1\otimes n_1
		+ m_1\otimes [[{}^mn, n'], n_1]                                                                                                        \\
		                                     & = {}^{[{}^mn, n']}m_1\otimes n_1+ {}^{n_1}m_1\otimes [{}^mn, n']-{}^{[{}^mn, n']}m_1\otimes n_1 \\
		                                     & = {}^{n_1}m_1\otimes [{}^mn, n']
	\end{align*}
	while, on the other hand,
	\begin{align*}
		[m_1\otimes n_1,{}^{n'}(m\otimes n)] & =[m_1\otimes n_1, {}^{n'} m\otimes n+m\otimes [n',n]]              \\
		                                     & =
		[m_1\otimes n_1, {}^{n'} m\otimes n+{}^nm\otimes n'-{}^{n'} m\otimes n]                                   \\
		                                     & =[m_1\otimes n_1, {}^nm\otimes n']=-{}^{n_1}m_1 \otimes {}^{^nm}n' \\
		                                     & = -{}^{n_1}m_1 \otimes {}^{^nm}n' =
		-{}^{n_1}m_1 \otimes [n', {}^mn]                                                                          \\
		                                     & ={}^{n_1}m_1 \otimes [{}^mn, n'].
	\end{align*}
	Moreover,
	\begin{align*}
		^{^{n_1}(m\otimes n)}n_2 & = {}^{^{n_1}m\otimes n + m\otimes [n_1, n]}n_2 = [{}^{n_1}m\otimes n + m\otimes [n_1, n], n_2] \\
		                         & =[{}^{^{n_1}m} n, n_2]+[{}^m [n_1, n], n_2]                                                    \\
		                         & =[[n, {}^m{n_1}], n_2]+[([{}^mn_1, n]+[n_1, {}^mn]), n_2 ]                                     \\
		                         & =[[n_1, {}^mn], n_2 ] = [n_2, [{}^mn, n_1]] =[n_2, {}^{m\otimes n}n_1],
	\end{align*}
	which finishes the proof.
\end{proof}

Thanks to this lemma we can consider the tensor product $(\mm \otimes \nn)\otimes \nn$ with its canonical homomorphism $(\mm \otimes \nn)\otimes \nn \to \nn$, which will again be a crossed module. Inductively, we obtain a crossed module
$(\cdots ((\mm \otimes \nn)\otimes \nn )\otimes \cdots )\otimes \nn \to \nn$. If $\mm = \nn$, then we introduce the following notations:
$$
	\mm^{\otimes 2}=\mm \otimes \mm, \qquad \mm^{\otimes i+1}=(\mm^{\otimes i}) \otimes \mm, \; i\geq 1.
$$
We denote by $\nu_{i}$ the homomorphism $\mm^{\otimes i}\to \mm$ described above. An explicit formula for $\nu_i$ on
generators is
$$
	\nu_{i} (\cdots ((m_1 \otimes m_2)\otimes m_3 )\otimes \cdots \otimes m_{i})=
	[\cdots [[m_1, m_2], m_3 ], \dots , m_{i}] .
$$

Denote by $\mu_\nn (\mm)$ the image of $\mu\colon \mm \otimes \nn \to \mm$. Since $\mu_\nn (\mm)$ is closed under the action of $\nn$,
we can form $\mu_\nn (\mu_\nn (\mm))$. By the same argument we can form $\mu_{\nn} ( \mu_\nn (\mu_\nn (\mm)) )$ and so on.
Set $\mu_\nn^0(\mm)=\mm$, $\mu_\nn^{i+1}(\mm)=\mu_\nn ( \mu_\nn^{i}(\mm) )$, $i\geq 0$. Since $\mu_\nn^{i}(\mm)\otimes \nn$ and $\nn$
act on each other compatibly, we can consider $\mu_\nn^{j}(\mu_\nn^{i}(\mm)\otimes \nn )$, for each $i$, $j \geq 0$.

\begin{lemma}\label{lemma1} We have
	$$
		\mu_\nn^{j}(\mu_\nn^{i}(\mm)\otimes \nn )= \Img \bigl( \mu_\nn^{i+j}(\mm)\otimes \nn \to \mu_\nn^{i}(\mm)\otimes \nn\bigr),
	$$
	where $\mu_\nn^{i+j}(\mm)\otimes \nn \to \mu_\nn^{i}(\mm)\otimes \nn$ is the homomorphism induced by the natural inclusion
	$\mu_\nn^{i+j}(\mm)\to \mu_\nn^{i}(\mm)$ for each $i$, $j \geq 0$.
\end{lemma}
\begin{proof} We have
	$$
		^{n'}(m\otimes n) = {}^{n'}m \otimes n + m\otimes [n', n]= {}^{n'}m \otimes n + {}^n{m}\otimes n' -{}^{n'}m\otimes n = {}^n{m}\otimes n',
	$$
	for each $m\in \mu_\nn^{i}(\mm)$ and $n$, $n'\in \nn$. Since $^n{m}\in \mu_\nn^{i+1}(\mm)$, we get
	$$
		\mu_\nn^{1}(\mu_\nn^{i}(\mm)\otimes \nn )= \Img \bigl( \mu_\nn^{i+1}(\mm)\otimes \nn \to \mu_\nn^{i}(\mm)\otimes \nn\bigr),
	$$
	where $\mu_\nn^{i+1}(\mm)\otimes \nn \to \mu_\nn^{i}(\mm)\otimes \nn$ is naturally induced as above. By induction we obtain the desired result.
\end{proof}

Given a Lie algebra $\mm$, recall that the lower central series $\gamma_1(\mm)$, $\gamma_2(\mm)$, \dots\ is defined by
$\gamma_1(\mm)=\mm$, $\gamma_{i+1}=[\gamma_i(\mm), \mm]$, $i\geq 1$. Observe that the homomorphism
$\nu_{i+1}\colon \mm^{i+1}\to \mm$ defined above is such that $\gamma_{i+1}(\mm)=\Img(\nu_{i+1})$.

\begin{theorem}\label{ThmTensorGamma}
	Let $\mm$ be finitely generated Lie algebra and $i$ be a fixed natural number.
	Then the following are equivalent:
	\begin{tfae}
		\item $\gamma_{i+1}(\mm)$ is finitely generated;
		\item $\mm^{\otimes i+1}$ is finitely generated.
	\end{tfae}
\end{theorem}
\begin{proof} The implication (ii)$\Rightarrow$(i) is obvious, so let us show the converse.
	We will prove that $\mm^{\otimes s}$ is finitely generated for all $1\leq s \leq i+1$, where $\mm^{\otimes 1}=\mm$.
	By Theorem \ref{theorem2} it is enough to prove that $\mu_{\mm^{\otimes s}}(\mm)$ and $\mu_\mm (\mm^{\otimes s})$ are
	finitely generated for all $1\leq s \leq i$. Note that $\mu_{\mm^{\otimes s}}(\mm)=\gamma_{s+1}(\mm)$. Moreover, since
	$\gamma_s(\mm)/\gamma_{s+1}(\mm)$ is finitely generated for all $s\geq 1$, $\gamma_s(\mm)$ will be finitely generated
	for all $1\leq s \leq i+1$. Thus, it suffices to show that $\mu_\mm (\mm^{\otimes s})$ is finitely generated for all
	$1\leq s \leq i$.
	
	We will show that $\mu_\mm^t (\mm^{\otimes s})$ is finitely generated for all $t$ and $s$ such that $1\leq t$, $s \leq i$
	and $s+t\leq i+1$. If $s=1$, then it is true. Assume that the same is true for a fixed integer $s$. Using Lemma~\ref{lemma1}
	we have:
	$$
		\mu_{\mm}^t(\mm^{\otimes s+1}) = \mu_{\mm}^t(\mm^{\otimes s}\otimes \mm) =
		\Img \bigl(\mu_{\mm}^t(\mm^{\otimes s})\otimes \mm \to \mm^{\otimes s}\otimes \mm \bigr).
	$$
	If $t+s\leq i$, then by Theorem \ref{theorem2}, the algebra $\mu_{\mm}^t(\mm^{\otimes s})\otimes \mm$ will be finitely generated, because
	$$
		\mu_{\mu_{\mm}^t(\mm^{\otimes s})}(\mm)=\gamma_{s+t+1}(\mm) \qquad \text{and}\qquad \mu_{\mm}\big(\mu_{\mm}^t(\mm^{\otimes s})\big)=
		\mu_{\mm}^{t+1}(\mm^{\otimes s})
	$$
	are finitely generated. Hence, $\mu_{\mm}^t(\mm^{\otimes s+1})$ is finitely generated for all $t$, $s$ with $t+s\leq i$.
	This completes the proof.
\end{proof}

Let $\ll$ be a Lie algebra and $\rr$ an ideal of $\ll$. We define the Lie algebras $\gamma_i(\rr, \ll)$,
$i\geq 1$, by $\gamma_1(\rr, \ll)=\rr$, $\gamma_{i+1}(\rr, \ll)=
	[\gamma_i(\rr, \ll), \ll]$.

\begin{definition} We say that an extension of Lie algebras $0 \to \rr \to \ll \to \mm \to 0$ is \emph{$i$-central},
	if $\gamma_{i+1}(\rr, \ll)=0$.
\end{definition}

\begin{theorem} \label{theorem3.11}
	Let $\mm$ be a finitely generated Lie algebra and $i$ a natural number. Then the following are equivalent:
	\begin{tfae}
		\item $\gamma_{i+1}(\mm)$ is finitely generated;
		\item there exists an $i$-central extension $0 \to \rr \to \ll \to \mm \to 0$ of Lie algebras, where $\gamma_{i+1}(\ll)$ is finitely generated;
		\item for every $i$-central extension $0 \to \rr \to \ll \to \mm \to 0$ of Lie algebras, $\gamma_{i+1}(\ll)$ is finitely generated.
	\end{tfae}
\end{theorem}
\begin{proof}
	The implications (iii)$\Rightarrow$(ii)$\Rightarrow$(i) are obvious, so let us show that (i) implies (iii). Since non-abelian tensor products are right exact, we have the short exact sequence of Lie algebras
	$$
		0 \to \widetilde{\ll^{\otimes i+1}} \to \ll^{\otimes i+1} \to \mm^{\otimes i+1} \to 0
	$$
	where $\widetilde{\ll^{\otimes i+1}}$ is the ideal of $\ll^{\otimes i+1}$ generated by all ${(\cdots ((l_1\otimes l_2)\otimes l_3)\otimes \cdots )\otimes l_{i+1}}$ such that $l_s\in \rr$ for at least one $s$. Therefore,
	$$
		\nu^\ll_{i+1}\big( \widetilde{\ll^{\otimes i+1}} \big) = 0.
	$$
	Since $\nu^\ll_{i+1}( \ll^{\otimes i+1} )= \gamma_{i+1}(\ll)$, we get a surjective homomorphism $\mm^{\otimes i+1} \to \gamma_{i+1}(\ll)$.
	The previous theorem completes the proof.
\end{proof}

\section{On Finitely Presented Lie Algebras}\label{S:LieFinPres}

Let $\mm$ be a Lie algebra over a field $\K$ and consider $\nu_i\colon \mm^{\otimes i} \to \gamma_i(\mm)$, defined as in the previous section.

\begin{theorem}\label{theorem4.1}
	If $\mm$ a finitely presented $\K$-Lie algebra and $i\geq 2$, then the kernel of $\nu_i\colon \mm^{\otimes i} \to \gamma_i(\mm)$ is finite-dimensional over $\K$.
\end{theorem}
\begin{proof}
	First we check that $\Ker (\nu_2)$ is finitely presented. Let $\mm \wedge \mm$ denote the non-abelian exterior product of $\mm$, which is the quotient of $\mm \otimes \mm$
	over the ideal generated by all the elements of the form $m\otimes m$ where $m\in \mm$. Let $m\wedge m'$ denote the class of $m\otimes m'$ in $\mm \wedge \mm$. Consider the morphism ${\mm \wedge \mm \to \mm \colon m\wedge m' \mapsto [m, m']}$. Its kernel is isomorphic to $H_2(\mm)$---the second homology of $\mm$ with trivial coefficients~\cite{Ellis1987}. Since $\mm$ is finitely presented, the Hopf formula implies that the $\K$-vector space $H_2(\mm)$ is finite-dimensional. On the other hand, by \cite{El0} we have an exact sequence
	$$
		\Gamma(\mm^{\ab}) \to \mm \otimes \mm \to \mm \wedge \mm \to 0,
	$$
	where $\Gamma$ denotes Whitehead's universal quadratic functor \cite{El0}. Since $\mm^{\ab}$ is finite-dimensional, $\Gamma(\mm^{\ab})$ is also finite-dimensional~\cite{SimsonTyc}. The exact sequence
	$$
		\Gamma(\mm^{\ab}) \to \Ker \bigl( \nu_{2}\colon \mm \otimes \mm \to \mm \bigr) \to H_2(\mm) \to 0
	$$
	now implies that $\Ker (\nu_2)$ is finite-dimensional.
	
	We proceed by induction on $i$, assuming that the claim is true for~$\nu_i$. Since the tensor product is a right exact functor, the extension of Lie algebras
	$$
		1 \to \Ker (\nu_i) \to \mm^{\otimes i} \to \gamma_i(\mm) \to 0
	$$
	yields the exact sequence
	$$
		\Ker( \nu_i)\otimes \mm \to \mm^{\otimes i}\otimes \mm \to \gamma_i(\mm) \otimes\mm \to 0.
	$$
	Denote the image of the homomorphism $\Ker( \nu_i)\otimes \mm \to \mm^{\otimes i}\otimes \mm$ by $\sigma \bigl( \Ker( \nu_i)\otimes \mm \bigr)$. Then we have the commutative diagram with exact rows
	\[
		\xymatrix{
		0 \ar[r] & \sigma \bigl( \Ker( \nu_i)\otimes \mm \bigr) \ar[d]_-{ } \ar[r]^-{ } & \mm^{\otimes i}\otimes \mm
		\ar[d]_-{\nu_{i+1}} \ar[r]^-{ } & \gamma_i(\mm)\otimes \mm \ar[d]^-{[-,-]} \ar[r] & 0 \\
		0 \ar[r] & 0 \ar[r]_-{} & \gamma_{i+1}(\mm) \ar[r]_-{1_{\gamma_{i+1}(\mm)}} & \gamma_{i+1}(\mm) \ar[r] & 0,
		}
	\]
	which yields the exact sequence
	$$
		0 \to \sigma ( \Ker( \nu_{i}) \otimes \mm) \to \Ker( \nu_{i+1}) \to
		\Ker \bigl( [-,-]\colon \gamma_n(\mm)\otimes \mm \to \gamma_{n+1}(\mm)\bigr) \to 0,
	$$
	where $[-,-]\colon \gamma_n(\mm)\otimes \mm \to \gamma_{n+1}(\mm)\colon\omega \otimes m \mapsto [\omega, m]$. We will show that the first and the last term in this exact sequence are finite-dimensional.
	
	We first prove that $\Ker(\nu_{i}) \otimes \mm$ is finite-dimensional.
	Let $v_1$, \dots, $v_k\in \mm$ be generators of the Lie algebra $\mm$.
	By the definition $\Ker(\nu_{i})\otimes \mm$
	is a quotient of $\Ker(\nu_{i})\otimes_{\K} \mm$. Using one of the defining relations, we have
	$$
		n\otimes [m, m']=[m',n]\otimes m-[m, n]\otimes m' ,
	$$
	for each $n\in \Ker(\nu_{i})$, $m\in \mm$. Therefore, as a vector space $\Ker(\nu_{i}) \otimes \mm$ will be a quotient of $\Ker(\nu_{i}) \otimes_\K V$, where $V$ is the subspace of $\mm$ generated by $v_1$, \dots, $v_k$. By the induction hypothesis $\Ker(\nu_{i})$ is finite-dimensional. Hence, $\Ker(\nu_{i}) \otimes_\K V$, and \emph{a fortiori} $\Ker(\nu_{i}) \otimes \mm$, is finite-dimensional.
	
	Now we show that $\Ker \bigl( [-,-]\colon \gamma_i(\mm)\otimes \mm \to \gamma_{i+1}(\mm)\bigr)$
	is finite-dimensional. By \cite{El0}, we have an exact sequence,
	$$
		\Gamma \big(\gamma_i(\mm)/\gamma_{i+1}(\mm)\big)\to \gamma_i(\mm)\otimes \mm 
		\to \gamma_i(\mm)\wedge \mm \to 0,
	$$
	which in turn entails the exact sequence
	\begin{align*}
		\Gamma \big(\gamma_i(\mm)/\gamma_{i+1}(\mm)\big) & \to \Ker \bigl( [-,-]\colon \gamma_i(\mm)\otimes \mm \to \gamma_{i+1}(\mm)\bigr)       \\
		                                                 & \to \Ker \bigl([-,-]\colon \gamma_i(\mm) \wedge \mm \to \gamma_{i+1}(\mm)\bigr) \to 0
	\end{align*}
	where $\Gamma$ denotes the universal quadratic functor. Since $\mm$ is finitely generated,
	$\gamma_i(\mm)/\gamma_{i+1}(\mm)$ will be finite-dimensional. Therefore,
	$\Gamma \big(\gamma_i(\mm)/\gamma_{i+1}(\mm)\big)$ is finite-dimensional.
	Thus, it suffices to show that $\Ker \bigl( [-,-]\colon \gamma_i(\mm)\wedge \mm \to \gamma_{i+1}(\mm)\bigr)$
	is finite-dimensional. For the latter, we will use the exact sequence
	$$
		H_3\bigl(\mm /\gamma_i(\mm)\bigr)\to
		\Ker \bigl( [-,-]\colon \mm \wedge \gamma_i(\mm)\to \gamma_{i+1}(\mm)\bigr) \to H_2(\mm),
	$$
	which follows from a combination of Theorem 34 and Theorem 35 (iii) in~\cite{El0}. We have already pointed out that $H_2(\mm)$ is finite-dimensional. Moreover, since $\mm /\gamma_i(\mm)$ is finite-dimensional, $H_3\big(\mm /\gamma_i(\mm)\big)$ is also finite-dimensional---use the Chevalley--Eilenberg chain complex of $\mm /\gamma_i(\mm)$. This finishes the proof.
\end{proof}

\begin{theorem}\label{ThmTensorGammaFinPres}
	Let $\mm$ be finitely presented Lie algebra and $i\geq 2$. Then $\mm^{\otimes i}$ is finitely presented if and only if $\gamma_i(\mm)$ is finitely presented.
\end{theorem}
\begin{proof} Consider the following extension of Lie algebras:
	$$
		0 \to \Ker (\nu_i) \to \mm^{\otimes i} \to \gamma_i(\mm) \to 0.
	$$
	By the previous theorem, the first term $\Ker (\nu_i)$ is finite-dimensional. Therefore, the middle term $\mm^{\otimes i}$ is finitely presented if and only if $\gamma_i(\mm)$ is
	finitely presented.
\end{proof}

\begin{theorem}\label{Thm4.3}
	Let $\mm$ be a finitely presented Lie algebra and $i$ a natural number. Then the following are equivalent:
	\begin{tfae}
		\item $\gamma_{i+1}(\mm)$ is finitely presented;
		\item there exists an $i$-central extension $0 \to \rr \to \ll \to \mm \to 0$ of Lie algebras, where $\gamma_{i+1}(\ll)$ is finitely presented;
		\item for every $i$-central extension $0 \to \rr \to \ll \to \mm \to 0$ of Lie algebras, $\gamma_{i+1}(\ll)$ is finitely presented.
	\end{tfae}
\end{theorem}
\begin{proof}
	As in Theorem \ref{theorem3.11}, we have the short exact sequence of Lie algebras
	$$
		0 \to \widetilde{\ll^{\otimes i+1}} \to \ll^{\otimes i+1} \to \mm^{\otimes i+1} \to 0,
	$$
	where $\widetilde{\ll^{\otimes i+1}}$ is the ideal of $\ll^{\otimes i+1}$ generated by all the elements
	$${(\cdots ((l_1\otimes l_2)\otimes l_3)\otimes \cdots )\otimes l_{i+1}}$$ such that $l_s\in \rr$ for at least one
	$s$. Since $\nu^\ll_{i+1}\big( \widetilde{\ll^{\otimes i+1}} \big) = 0$, the homomorphism $\nu^\ll_{i+1}\colon {\ll^{\otimes i+1}\to \gamma_{i+1}(\ll)}$ factors over the cokernel $\ll^{\otimes i+1} \to \mm^{\otimes i+1}$ and we obtain the dotted morphism $\mm^{\otimes i+1} \to \gamma_{i+1}(\ll)$ making the following diagram commute.
	\[
		\xymatrix{
		& \mm^{\otimes i+1} \ar@{.>}[d]_-{} \ar@{=}[r]^-{ } & \mm^{\otimes i+1} \ar[d]^-{\nu^\mm_{i+1}} \\
		& \gamma_{i+1}(\ll) \ar[r]_-{} & \gamma_{i+1}(\mm)
		}
	\]
	This implies that $\Ker \bigl( \mm^{\otimes i+1} \to \gamma_{i+1}(\ll)\bigr) \subseteq \Ker(\nu^\mm_{i+1})$.
	By Theorem~\ref{theorem4.1}, the Lie algebra $\Ker(\nu^\mm_{i+1})$ is finite-dimensional; hence so is $\Ker \bigl( \mm^{\otimes i+1} \to \gamma_{i+1}(\ll)\bigr)$. It follows that $\mm^{\otimes i+1}$ is finitely presented if and only if $\gamma_{i+1}(\ll)$ is finitely presented. On the other hand, by the previous theorem, $\mm^{\otimes i+1}$ is finitely presented if and only if $\gamma_{i+1}(\mm)$ is finitely presented.
\end{proof}

%\section*{Acknowledgements}

%\bibliography{tim}
%\bibliographystyle{amsplain-nodash}
\providecommand{\noopsort}[1]{}
\providecommand{\bysame}{\leavevmode\hbox to3em{\hrulefill}\thinspace}
\providecommand{\MR}{\relax\ifhmode\unskip\space\fi MR }
\providecommand{\MRhref}[2]{%
	\href{http://www.ams.org/mathscinet-getitem?mr=#1}{#2}
}
\providecommand{\href}[2]{#2}

\end{document}